\documentclass[12pt]{amsart}

\usepackage{amsmath,amssymb,amsbsy,amsfonts,amsthm,latexsym,
                     amsopn,amstext,amsxtra,euscript,amscd,mathrsfs}
\usepackage{amsmath,amssymb,amsbsy,amsfonts,latexsym,amsopn,amstext,cite,
                                               amsxtra,euscript,amscd,bm}
\usepackage{mathtools}
\usepackage{todonotes}
\usepackage{url}
\usepackage[colorlinks,linkcolor=blue,anchorcolor=blue,citecolor=blue,backref=page]{hyperref}

\renewcommand*{\backref}[1]{}
\renewcommand*{\backrefalt}[4]{%
    \ifcase #1 (Not cited.)%
    \or        (p.\,#2)%
    \else      (pp.\,#2)%
    \fi}

\begin{document}




\newfont{\teneufm}{eufm10}
\newfont{\seveneufm}{eufm7}
\newfont{\fiveeufm}{eufm5}
%
%
\newfam\eufmfam
 \textfont\eufmfam=\teneufm \scriptfont\eufmfam=\seveneufm
 \scriptscriptfont\eufmfam=\fiveeufm
%
%
\def\frak#1{{\fam\eufmfam\relax#1}}
%


\def\bbbr{{\rm I\!R}} 
\def\bbbm{{\rm I\!M}}
\def\bbbn{{\rm I\!N}} 
\def\bbbf{{\rm I\!F}}
\def\bbbh{{\rm I\!H}}
\def\bbbk{{\rm I\!K}}
\def\bbbp{{\rm I\!P}}
\def\bbbone{{\mathchoice {\rm 1\mskip-4mu l} {\rm 1\mskip-4mu l}
{\rm 1\mskip-4.5mu l} {\rm 1\mskip-5mu l}}}
\def\bbbc{{\mathchoice {\setbox0=\hbox{$\displaystyle\rm C$}\hbox{\hbox
to0pt{\kern0.4\wd0\vrule height0.9\ht0\hss}\box0}}
{\setbox0=\hbox{$\textstyle\rm C$}\hbox{\hbox
to0pt{\kern0.4\wd0\vrule height0.9\ht0\hss}\box0}}
{\setbox0=\hbox{$\scriptstyle\rm C$}\hbox{\hbox
to0pt{\kern0.4\wd0\vrule height0.9\ht0\hss}\box0}}
{\setbox0=\hbox{$\scriptscriptstyle\rm C$}\hbox{\hbox
to0pt{\kern0.4\wd0\vrule height0.9\ht0\hss}\box0}}}}
\def\bbbq{{\mathchoice {\setbox0=\hbox{$\displaystyle\rm
Q$}\hbox{\raise
0.15\ht0\hbox to0pt{\kern0.4\wd0\vrule height0.8\ht0\hss}\box0}}
{\setbox0=\hbox{$\textstyle\rm Q$}\hbox{\raise
0.15\ht0\hbox to0pt{\kern0.4\wd0\vrule height0.8\ht0\hss}\box0}}
{\setbox0=\hbox{$\scriptstyle\rm Q$}\hbox{\raise
0.15\ht0\hbox to0pt{\kern0.4\wd0\vrule height0.7\ht0\hss}\box0}}
{\setbox0=\hbox{$\scriptscriptstyle\rm Q$}\hbox{\raise
0.15\ht0\hbox to0pt{\kern0.4\wd0\vrule height0.7\ht0\hss}\box0}}}}
\def\bbbt{{\mathchoice {\setbox0=\hbox{$\displaystyle\rm
T$}\hbox{\hbox to0pt{\kern0.3\wd0\vrule height0.9\ht0\hss}\box0}}
{\setbox0=\hbox{$\textstyle\rm T$}\hbox{\hbox
to0pt{\kern0.3\wd0\vrule height0.9\ht0\hss}\box0}}
{\setbox0=\hbox{$\scriptstyle\rm T$}\hbox{\hbox
to0pt{\kern0.3\wd0\vrule height0.9\ht0\hss}\box0}}
{\setbox0=\hbox{$\scriptscriptstyle\rm T$}\hbox{\hbox
to0pt{\kern0.3\wd0\vrule height0.9\ht0\hss}\box0}}}}
\def\bbbs{{\mathchoice
{\setbox0=\hbox{$\displaystyle     \rm S$}\hbox{\raise0.5\ht0\hbox
to0pt{\kern0.35\wd0\vrule height0.45\ht0\hss}\hbox
to0pt{\kern0.55\wd0\vrule height0.5\ht0\hss}\box0}}
{\setbox0=\hbox{$\textstyle        \rm S$}\hbox{\raise0.5\ht0\hbox
to0pt{\kern0.35\wd0\vrule height0.45\ht0\hss}\hbox
to0pt{\kern0.55\wd0\vrule height0.5\ht0\hss}\box0}}
{\setbox0=\hbox{$\scriptstyle      \rm S$}\hbox{\raise0.5\ht0\hbox
to0pt{\kern0.35\wd0\vrule height0.45\ht0\hss}\raise0.05\ht0\hbox
to0pt{\kern0.5\wd0\vrule height0.45\ht0\hss}\box0}}
{\setbox0=\hbox{$\scriptscriptstyle\rm S$}\hbox{\raise0.5\ht0\hbox
to0pt{\kern0.4\wd0\vrule height0.45\ht0\hss}\raise0.05\ht0\hbox
to0pt{\kern0.55\wd0\vrule height0.45\ht0\hss}\box0}}}}
\def\bbbz{{\mathchoice {\hbox{$\sf\textstyle Z\kern-0.4em Z$}}
{\hbox{$\sf\textstyle Z\kern-0.4em Z$}}
{\hbox{$\sf\scriptstyle Z\kern-0.3em Z$}}
{\hbox{$\sf\scriptscriptstyle Z\kern-0.2em Z$}}}}
\def\ts{\thinspace}

\newtheorem{thm}{Theorem}
\newtheorem{lem}{Lemma}
\newtheorem{lemma}[thm]{Lemma}
\newtheorem{prop}{Proposition}
\newtheorem{proposition}[thm]{Proposition}
\newtheorem{theorem}[thm]{Theorem}
\newtheorem{cor}[thm]{Corollary}
\newtheorem{corollary}[thm]{Corollary}

\newtheorem{prob}{Problem}
\newtheorem{problem}[prob]{Problem}
\newtheorem{ques}{Question}
\newtheorem{question}[ques]{Question}


\numberwithin{equation}{section}
\numberwithin{thm}{section}

\def\squareforqed{\hbox{\rlap{$\sqcap$}$\sqcup$}}
\def\qed{\ifmmode\squareforqed\else{\unskip\nobreak\hfil
\penalty50\hskip1em\null\nobreak\hfil\squareforqed
\parfillskip=0pt\finalhyphendemerits=0\endgraf}\fi}

\def\cA{{\mathcal A}}
\def\cB{{\mathcal B}}
\def\cC{{\mathcal C}}
\def\cD{{\mathcal D}}
\def\cE{{\mathcal E}}
\def\cF{{\mathcal F}}
\def\cG{{\mathcal G}}
\def\cH{{\mathcal H}}
\def\cI{{\mathcal I}}
\def\cJ{{\mathcal J}}
\def\cK{{\mathcal K}}
\def\cL{{\mathcal L}}
\def\cM{{\mathcal M}}
\def\cN{{\mathcal N}}
\def\cO{{\mathcal O}}
\def\cP{{\mathcal P}}
\def\cQ{{\mathcal Q}}
\def\cR{{\mathcal R}}
\def\cS{{\mathcal S}}
\def\cT{{\mathcal T}}
\def\cU{{\mathcal U}}
\def\cV{{\mathcal V}}
\def\cW{{\mathcal W}}
\def\cX{{\mathcal X}}
\def\cY{{\mathcal Y}}
\def\cZ{{\mathcal Z}}
\def\pgdc{\textrm{gcd}}
\newcommand{\rmod}[1]{\: \mbox{mod} \: #1}

\def\Nm{{\mathrm{Nm}}}

\def\Tr{{\mathrm{Tr}}}

\def\epp{\mathbf{e}_{p-1}}

\def\ind{\mathop{\mathrm{ind}}}

\def\mand{\qquad \mbox{and} \qquad}

\newcommand{\commH}[1]{\marginpar{%
\begin{color}{magenta}
\vskip-\baselineskip 
\raggedright\footnotesize
\itshape\hrule \smallskip H: #1\par\smallskip\hrule\end{color}}}

\newcommand{\commI}[1]{\marginpar{%
\begin{color}{blue}
\vskip-\baselineskip 
\raggedright\footnotesize
\itshape\hrule \smallskip I: #1\par\smallskip\hrule\end{color}}}



\newcommand{\ignore}[1]{}

\hyphenation{re-pub-lished}

\parskip 1.5 mm
\parindent 8 pt

\def\GL{\operatorname{GL}}
\def\SL{\operatorname{SL}}
\def\PGL{\operatorname{PGL}}
\def\PSL{\operatorname{PSL}}
\def\li{\operatorname{li}}

\def\vec#1{\mathbf{#1}}

\def \F{{\mathbb F}}
\def \K{{\mathbb K}}
\def \Z{{\mathbb Z}}
\def \N{{\mathbb N}}
\def \Q{{\mathbb Q}}
\def \T{{\mathbb T}}
\def \C {{\mathbb C}}
\def \R{{\mathbb R}}
\def\Fp{\F_p}
\def \fp{\Fp^*}

\def \Rc{{\mathcal R}}
\def \Qc{{\mathcal Q}}
\def \Ec{{\mathcal E}}

\def \DN{D_N}
\def\va{\mbox{\bf{a}}}

\def\Kc{\,{\mathcal K}}
\def\Ic{\,{\mathcal I}}

\def\\{\cr}
\def\({\left(}
\def\){\right)}
\def\fl#1{\left\lfloor#1\right\rfloor}
\def\rf#1{\left\lceil#1\right\rceil}

\def\Ln#1{\mbox{\rm {Ln}}\,#1}

\def \nd {\, | \hspace{-1.2mm}/\,}

 \def\e{\mathbf{e}}

\def\ep{\mathbf{e}_p}
\def\eq{\mathbf{e}_q}

\def\wt#1{\mbox{\rm {wt}}\,#1}

\def\Mob{M{\"o}bius }


\title[M{\"o}bius Transformation 
and M{\"o}bius function]{Disjointness of the M{\"o}bius Transformation 
and M{\"o}bius Function}

\author[E. H. El Abdalaoui]{El Houcein El Abdalaoui}
\address{
Laboratoire de Math{\'e}matiques Rapha{\"e}l Salem, 
Universit{\'e} de Rouen Normandie, 
F76801 Saint-{\'E}tienne-du-Rouvray, France}
\email{elhoucein.elabdalaoui@univ-rouen.fr}

\author{Igor E. Shparlinski}
\address{School of Mathematics and Statistics, University of New South Wales,
 Sydney NSW 2052, Australia}
\email{igor.shparlinski@unsw.edu.au}


\date{}

\pagenumbering{arabic}

\begin{abstract}
We study the distribution of the sequence of elements of the discrete 
dynamical system 
generated by the  \Mob transformation  $x \mapsto (ax + b)/(cx + d)$ over a finite 
field of $p$ elements.. 
Motivated by a recent conjecture of P.~Sarnak, we obtain nontrivial estimates of exponential sums with 
such sequences that imply that trajectories of this dynamical system 
are disjoined with the \Mob function. 
\end{abstract}

\maketitle

\section{Introduction}


Let, as usual $\mu(n)$ denote the  {\it \Mob function\/}, that is,
$\mu(n) = 0$ if $n$ is not squarefree and $\mu(n) = (-1)^s$ 
if $n$ is a product of $s$ distinct primes. 
Furthermore, given a 
  compact topological space $X$ and a homeomorphism  $T: X\to X$, we
consider the {\it flow} $\cX = (T,X)$.
The {\it \Mob disjointness conjecture\/}
of Sarnak~\cite{Sarn} asserts that for any flow $\cX = (T,X)$ 
of topological  entropy zero, 
we have 
\begin{equation}
\label{eq:DisjConj}
\sum_{n \le N} \mu(n) f\(T^n x\) = o(N), \qquad N \to \infty,
\end{equation}
for any $x \in X$ and 
a continuous complex-valued function $f$ on $X$. 
This conjecture has recently attracted very  active interest 
and has actually been established for several 
classes of flows, see~\cite{eALdlR,Bour3,BSZ,CarRud,Eis,FoGang,GrTao,K-PLe,LiSa,Ryz, FKL}
and references therein. Moreover, for the connection between the Sarnak and Chowla conjectures, we refer to
very recent works of el Abdalaoui~\cite{elabdal}, Gomilko, Kwietniak and Lema\'{n}czyk~\cite{lem},  Tao~\cite{Tao1} and 
Tao and Ter\"{a}v\"{a}inen~\cite{TaoTer}.

As usual, we use $\F_q$ to denote the finite field of $q$ elements.

Here we consider a discrete analogue of this conjecture for the 
flow $\cM = (A, \F_p)$ formed by the
{\it \Mob map\/}
\begin{equation}
\label{eq:MobMap}
A: \ x \mapsto \frac{ax+b}{cx+d}
\end{equation}
over   $\F_p$, where $p$ is
a sufficiently large prime.  

This transformation, over the complex numbers and also in finite fields and rings,  
has been extensively studied because of their relevance to the ergodic theory, and
dynamical systems, see~\cite{BNR, Bour2.5, Kelm, Kurl, KRR, KuRu1,  KuRu2, Rosen} and 
references therein. We also recall that this transformation on the 2-dimensional torus 
$\T^2 = (\R/\Z)^2$ is also sometimes called the {\it cat map\/}. It also has important links 
with theoretical physics, see, for example,~\cite[Section~4.3]{BluRe} or~\cite[Appendix~A]{MarWre}. 

Here we aim to establish  an appropriate version of the Sarnak conjecture for the
  \Mob map over $\F_p$. This leads us to investigating of exponential sums along the trajectories of~\eqref{eq:MobMap}
twisted with the \Mob function.

In turn, exponential  sums with the \Mob function  are 
closely related to sums over primes, which is associated with 
the behaviour of dynamical systems at  ``prime'' times, see~\cite{SarUb}
for a general point of view and also specific results 
for dynamical systems on $\SL_2(\R)$. 

We note  that  the study of ergodic dynamical system along  sequences of arithmetic interest, initiated by Bourgain~\cite{Bour0,Bour1,Bour1.5,Bour1.7}, see also 
 the surveys by Rosenblatt and Wierdl~\cite{RosenblattW} and  by Thouvenot~\cite{ThouvenotB}. 
In particular, the case of primes takes its origin in the works of  Bourgain~\cite{Bour1}
and Weirdl~\cite{Wierdl}; we refer also to the results of Nair~\cite{Nair1, Nair2}.
For several  more results on the Prime Ergodic Theorem
and Ergodic Theorem with Arithmetical Weights, we refer to~\cite{Buczolich, Cuny-Weber, elabdal-M-R, Eis, EisLin, Nair1, Nair2}, 
see also  the references therein and a very recent survey by   Eisner and  Lin~\cite{EisLin}. 

For the orbits of the one-dimensional dynamical system such as 
 $x \mapsto gx$ over  $\F_p$ such results are given in~\cite{BCFS,BFGS2, BFGS3, Bour2,GarShp,OstShp1}.

\section {Formal set-up}

For $c\ne 0$ also extend the definition~\eqref{eq:MobMap}
by setting
\begin{equation}
\label{eq:MobPole}
A(-d/c) = a/c.
\end{equation}
It is now easy to check that  this  extended map
 $x \mapsto A(x)$ induces a permutation 
of $\F_p$.

In fact, we always identify the map~\eqref{eq:MobMap}
with a nonsingular matrix
$$
A = \begin{pmatrix} a & b\\ c & d \end{pmatrix} \in \GL_2(\F_p), 
$$
and we also always assume that $c\ne 0$ (so $A$ is not a linear map).

Moreover, after an appropriate scaling of the coefficient of $A$ we 
can always assume that 
\begin{equation}
\label{eq:MatrM}
A = \begin{pmatrix} a & b\\ c & d \end{pmatrix} \in \SL_2(\F_p). 
\end{equation}

Furthermore, for $\xi_0\in \F_p$ we consider the trajectory 
\begin{equation}
\label{eq:Traj}
\xi_{n} =A\(\xi_{n-1}\) = A^n\(\xi_0\),
 \qquad n =  1,2,
\ldots\,,
\end{equation}
generated  by iterations of $A$. 

It is easy to see that each sequence of the form~\eqref{eq:Traj}
either terminates after finitely many steps (if  $c \xi_{n-1}+d = 0$)
of is eventually periodic, and then, as $A$ is a permutation it is 
purely periodic.

It is known that showing~\eqref{eq:DisjConj}
can be reduced to estimating exponentials sums along trajectories of $\cX$
twisted by the \Mob function. In our case, we are interested in the sums
\begin{equation}
\label{eq:Sum S}
S_\psi(N) = \sum_{n \le N} \mu(n) \psi\(\xi_n\)
\end{equation}
twisted by the \Mob function 
along the trajectory~\eqref{eq:Traj} 
with a nontrivial additive character $\psi$ of $\F_p$.

We remark, that similar sums, however associated with 
a linear  map $x \mapsto g x$ over $\F_p$, that is, 
of the sequence $\xi_0g^n$, have been estimated 
in~\cite[Theorem~5.1]{BCFS}. 
In fact, using the ideas of~\cite{BFGS3} it is possible to 
improve~\cite[Theorem~5.1]{BCFS}, see also~\cite{Bour2}. Furthermore, 
exponential sums over primes, associated with similar dynamical systems
on elliptic curves over $\F_p$ have been estimated in~\cite{BFGS2} 
(see also~\cite[Section~4]{OstShp1}), and can easily be extended to
sums with the \Mob function.

\section {Our approach and main result}

One of the ingredients of our approach is an explicit formula for the elements 
of the sequence~\eqref{eq:Traj}, which is essentially based on the spectral decomposition, 
see  Lemma~\ref{lem:ExpFrac}. 

We then combine it with a  new general result which holds for arbitrary multiplicative 
functions and which is based on the ideas of Bourgain, Sarnak
and Ziegler~\cite[Theorem~2]{BSZ} and    K{\'a}tai~\cite{Kat}, see Lemma~\ref{lem:KBSZ}, which we believe is
of independent interest.

This is  combined with some bounds of exponential sums   which in turn are 
 based on the Weil bound, see Lemmas~\ref{lem:Weil} and~\ref{lem:Li} and  allows us
to obtain a non-trivial estimate for the sums~\eqref{eq:Sum S}.

Throughout the paper, the implied constants in 
the symbols `$O$',  `$\ll$' and   `$\gg$' 
may occasionally, where obvious, depend on 
the  real positive parameter   $\varepsilon$, and 
are absolute otherwise
(we recall that $U \ll V$ and  $V \gg U$  are both equivalent
to $U = O(V)$).

In all our bounds we have to assume that 
\begin{equation}
\label{eq: t large}
t \ge p^{1/2+\varepsilon}, 
\end{equation}
which is not a severe restriction as it is satisfied 
by the majority of the sequences, see, for example,~\cite{Chou}. 

Our main result is the following bound:

\begin{theorem}
\label{thm:MobSum}  
Let     $\varepsilon>0$ be  a fixed  sufficiently small real number. 
If  the characteristic polynomial of the matrix $A$  of the form~\eqref{eq:MatrM}
has two distinct roots  in $\F_{p^2}$ and
the period length  $t$ of the sequence~\eqref{eq:Traj}  satisfies~\eqref{eq: t large} 
then,   for any real    $\alpha$ with
\begin{equation}
\label{eq: Cond_T1}
\alpha \ge  \frac{3 ( \log \log p)^6}{\varepsilon  \log p}
\end{equation}
and integer  $N$ with
\begin{equation}
\label{eq: Cond_T2}
N  \ge  p^{1/2}  \exp\(5 \alpha^{-1} \(\log (1/\alpha)\)^6\) \log p.  
\end{equation}
uniformly over all  nontrivial  additive characters $\psi$ of $\F_p$, we have
$$
|S_\psi(N)| \ll   \alpha N.
$$
\end{theorem}

We  remark that  Theorem~\ref{thm:MobSum}  is nontrivial starting from the values of $N$ slighly 
larger than $p^{1/2}$ which is certainly the best possible range until the  
the condition~\eqref{eq: t large} is relaxed (which is presently necessary for nontrivial estimates
of exponential sums along consecutive integers).  
\section{\Mob transformation 
and binary recurrences} 
\label{lin:LRS}

\begin{lemma}
\label{lem:LRS} 
Let $f(Z)= Z^2 - eZ +1 \in \F_p[Z]$, where $e = a+d$, 
be the characteristic polynomial of the matrix $A$ of the form~\eqref{eq:MatrM}. 
Then there are two binary recurrence sequences $u_n$ and $v_n$ satisfying
$$
u_{n+2} = e u_{n+1} + u_n \mand v_{n+2} = ev_{n+1} + v_n 
$$
with  the initial values 
$$
(u_0, u_1) = (\xi_0, a\xi_0 + b) \mand (v_0, v_1) = (1, c\xi_0 + d) 
$$
such that 
$$
\xi_n = u_n/v_n
$$
for $n =0,1, \ldots$. 
\end{lemma}

\begin{proof} It is easy to check that the recursive definition of $u_n$ and $v_n$ can be rewritten as 
$$
 \begin{pmatrix} u_{n+1}\\ v_{n+1} \end{pmatrix}
  = A   \begin{pmatrix} u_n\\ v_n \end{pmatrix} , \qquad n = 0,1, \ldots. 
$$
with  the initial values 
$$
(u_0, u_1) = (\xi_0, a\xi_0 + b) \mand (v_0, v_1) = (1, c\xi_0 + d) 
$$
Then one verifies that the desired statement by induction on $n$. 
\end{proof}

Using the well known expression of linear recurrence sequences via the roots of 
characteristic polynomials, see, for example,~\cite{EvdPSW}, we immediately
derive from Lemma~\ref{lem:LRS}, in a straightforward fashion, the following explicit 
formula:

\begin{lemma}
\label{lem:ExpFrac} Let $f(Z)= Z^2 - eZ +1 \in \F_p[Z]$, where $e = a+d$, 
be the characteristic polynomial of the matrix $A$  of the form~\eqref{eq:MatrM}, 
 which has two distinct roots $\vartheta$ and 
$\vartheta^{-1}$ in $\F_{p^2}$. 
Then there exist elements $\alpha, \beta, \gamma\in \F_{p^2}$
such that 
$$
\xi_n = \alpha + \frac{\beta}{\vartheta^{2n} + \gamma},  \qquad n = 0,1, \ldots. 
$$
\end{lemma}

 \section{Bounds on single character sums} 

Let $p$ be the characteristic of $\F_p$ and let 
$\overline \F_p$ denote the algebraic closure of $\F_p$. 
For an exhasutive account on the character sums over finite fields we refer to~\cite[Chapter~11]{IwKow}. 

One of our main tools is the bound on hybrid sums 
of multiplicative and additive characters, which 
in its classical form is  given by Weil~\cite[Example~12 of Appendix~5]{Weil}; 
see also~\cite[Theorem~3 of Chapter~6]{Li1}.  

\begin{lemma}
\label{lem:Weil}
For any polynomials $g(X), h(X) \in \F_p[X]$
and  any nontrivial  additive character $\psi$  and arbitrary 
multiplicative character $\chi$ of $\F_p$ we have
$$
 \sum_{\substack{x \in \F_p \\ g(x) \ne 0}}
\psi\( h(X)/g(X)\)\chi(x)  \ll  \max\{\deg g\,, \deg h\}  p^{1/2}.  $$
\end{lemma}

We also need a modification of a bound of Li~\cite[Theorem~2]{Li2}
which also applies to rational functions rather than to polynomials.
We recall that the trace map and norm map from $\F_{p^n}$ to $\F_p$ are given by 
$$\Tr_{\F_{p^n}/\F_p(z)}=\sum_{i=0}^{n-1}z^i  \mand
\Nm_{\F_{p^n}/\F_p(z)}=\prod_{i=0}^{n-1}z^{p^i}.$$

To simplify the notation, we use $\Tr(z)$ and $\Nm(z)$ to denote the trace and the norm, respectively,  
of an element  $z$ of the quadratic extension  $\F_{p^2}$ in $\F_p$. 
Combining the arguments of the proof of~\cite[Theorem~2]{Li2} with those
in~\cite[Section~4]{Li15}, one easily derives:

\begin{lemma}
\label{lem:Li}
For any polynomials $g(X), h(X) \in \F_{p^2}[X]$
and  any nontrivial  additive character $\psi$  and arbitrary 
multiplicative character $\chi$ of $\F_p$ we have
$$
 \sum_{\substack{x \in \F_{p^2} \\ g(x) \ne 0\\\Nm(x) =1}}
\psi\( \Tr\(h(X)/g(X)\)\)\chi(x)  \ll  \max\{\deg g\,, \deg h\}  p^{1/2}.  
$$
\end{lemma}

Note that Lemma~\ref{lem:Li} can be further extended in several 
directions.

We now need a bound on the character sums 
$$
Q_{\psi}(u,v;k,m,N) = \sum_{n \le N} \psi\(u\xi_{kn}+v\xi_{mn}\)
$$
with $u,v \in \F_p$ and non-negative integers $k$ and $m$, 
along consecutive values of the trajectory~\eqref{eq:Traj}.

\begin{lemma}
\label{lem:SinglSum 2term} Assume that the characteristic polynomial of the matrix $A$ 
of the form~\eqref{eq:MatrM} has two distinct roots in $\F_{p^2}$. 
If  $t$ is the period length of the sequence~\eqref{eq:Traj}, 
then, for any  $u,v \in \F_p$ with $(u,v) \ne (0,0)$ and  integers $0 \le k<m$, for any $N \le t$ we have 
$$
Q_{\psi}(u,v;k,m,N)  \ll   m  p^{1/2} \log p. 
$$
\end{lemma}

\begin{proof} Let $\vartheta$ be as in Lemma~\ref{lem:ExpFrac}. 
It is clear that $t$ is  the multiplicative order of $\vartheta^2$. 
For every  integer $h$, We now define the sums
$$
Q_{h,\psi}(u,v;k,m)= \sum_{n=1}^t  \psi\(u\xi_{kn}+v\xi_{mn}\)\e(hn/t), 
$$
where
$$
\e(z) = \exp(2 \pi i z).
$$ 

We first consider the case $\vartheta\in \F_p$. 
Since $\vartheta^2$ is of order $t$ it can be written as 
$\vartheta^2 = g^s$ for $s = (p-1)/t$ and some primitive root $g$ of 
$\F_p^*$. 
For $x \in \F_p^*$, we define $\ind x$ by the conditions
$$
g^{\ind x} = x \mand 0 \le \ind x \le p-2.
$$
Hence, using Lemma~\ref{lem:ExpFrac} and the additivity of $\psi$,  we write 
\begin{align*}
Q_{h,\psi}&(u,v;k,m) \\
& = 
 \psi\(\alpha\(u+v\)\) \sum_{n=1}^t \psi\(\frac{\beta u }{g^{kns} + \gamma}+  \frac{\beta v}{g^{mns} + \gamma}\) \e(h n/t)\\
& =  \psi\(\alpha\(u+v\)\) \sum_{n=1}^t \psi\(\frac{\beta u }{g^{kns} + \gamma}+  \frac{\beta v}{g^{mns} + \gamma}\)  \e(h sn/(p-1))\\
& = \frac{1}{s}   \psi\(\alpha\(u+v\)\) \sum_{n=1}^{p-1} \psi\(\frac{\beta u }{g^{kns} + \gamma}+  \frac{\beta v}{g^{mns} + \gamma}\) \e(h sn/(p-1)) .\end{align*}
Now, denote $x = g^n$ and using that $g$ is a primitive root, we obtain
\begin{align*}
Q_{\psi}&(h,u,v;k,m) \\& = \frac{1}{s}   \psi\(\alpha\(u+v\)\)  \sum_{x \in \F_p^*}   
 \psi\(\frac{\beta u }{x^{ks} + \gamma}+  \frac{\beta v}{x^{ms} + \gamma}\) \e(h s\ind x/(p-1)).
\end{align*}
Since the function $x \mapsto  \e(h s \ind x/(p-1))$ is a multiplicative character of $\F_p^*$, recalling 
Lemma~\ref{lem:Weil}, we obtain 
$$
Q_{h,\psi}(u,v;k,m) \ll  \frac{1}{s}sm p^{1/2} = m p^{1/2} .
$$
Using the standard reduction between complete and incomplete sums, 
see~\cite[Section~12.2]{IwKow}, we conclude the proof in this case.

If $\vartheta\in \F_{p^2}\setminus \F_p$, then obviously $\Nm(\vartheta) =1$. We now choose $g$ to get a generator of the 
norm group, consisting of the elements $z\in \F_{p^2}$ with $\Nm(z)=1$ and we proceed as
in he above, however using Lemma~\ref{lem:Li} instead of Lemma~\ref{lem:Weil} in the appropriate place. 
\end{proof}

For sums with one term 
$$
R_{\psi}(u;m,N) = \sum_{n \le N} \psi\(u\xi_{mn}\)
$$
with $u\in \F_p$ and a non-negative integer $m$,  
 we have a slightly more precise statement.

\begin{lemma}
\label{lem:SinglSum 1term} Assume that the characteristic polynomial of the matrix $A$ 
of the form~\eqref{eq:MatrM} has two distinct roots in $\F_{p^2}$. 
If  $t$ is the period length of the sequence~\eqref{eq:Traj}, 
then, for any  $u \in \F_p^*$   and  an integers $m>0$, for any $N \le t$ we have 
$$
R_{\psi}(u;m,N)  \ll   \gcd(m,t)  p^{1/2} \log p. 
$$
\end{lemma}

\begin{proof} Let $d = \gcd(m,t)$. We set 
$$
k = m/d \mand s = t/d.
$$
Let $B=A^k$. We also consider the sequences
$\zeta_n = \xi_{kn}$
then instead of~\eqref{eq:Traj}, we can write
$$
\zeta_{n} =B\(\zeta_{n-1}\) = B^n\(\xi_0\),
 \qquad n =  1,2,
\ldots\,.
$$
Hence,  by the standard arguments as before, we see that the period of the sequence $\zeta_n$ 
$n =  1,2, \ldots$, is $t$. 
Using a special case  (with only one term) applied to  $\zeta_{dn} = \xi_{mn}$ instead of $\xi_{n}$, 
we obtain the result. 
\end{proof}

\section{Double sums and correlations with multiplicative functions} 

Now, by applying the machinery in the proof of the criterion of Bourgain, Sarnak and Ziegler~\cite[Theorem~2]{BSZ} (which in turn improves the result of 
K{\'a}tai~\cite{Kat}), we obtain our main technical result.  

We present it in a form which is more general and flexible than we need here, since we believe it may 
find other applications. 

\begin{lemma}
\label{lem:KBSZ} Let $\nu$ be a multiplicative function and
$F$ an arbitrary periodic arithmetic function with period $t$. Assume 
$$
|\nu(n)| \leq 1 \mand  |F(n)| \leq 1, \qquad  n \in \N.
$$ 
 We further assume that for any primes $r \neq s$, and for any positive  integer $h \le t$ we have
$$
\left|\sum_{n \leq h} F(nr) \overline{F}(ns)\right| \ll 
  \max\{r,s\}  t \rho 
  $$
for some real $\rho < 1$. 
Then for any real $\alpha$ with  
\begin{equation}
\label{eq: Cond1}
 \alpha^{-2} \exp\( 2 \alpha^{-1} \(\log (1/\alpha)\)^6\)  \ll  \rho^{-1}
\end{equation}
and integer  $N$ with 
\begin{equation}
\label{eq: Cond2}
N  \gg  t \rho \exp\(4 \alpha^{-1} \(\log (1/\alpha)\)^6\), 
\end{equation}
 we have
$$\left|\sum_{ n \leq N}\nu(n)F(n)\right| \ll  \alpha  N. 
$$
\end{lemma}

\begin{proof} We follow  the proof of~\cite[Theorem~2]{BSZ}.
In particular, let $\alpha >0$  be  some sufficiently small (as otherwise there is
nothing to prove). 
As in~\cite[Equation~(2.1)]{BSZ} we define
\begin{equation}
\label{eq: jj}
j_0 = \frac{ \(\log (1/\alpha)\)^3}{ \alpha} \mand j_1 = j_0^2. 
\end{equation}
For every integer $j \in [j_0, j_1+1]$ we also define
$$
R_j =  (1+\alpha)^{j} \mand M_j = N/R_{j+1}. 
$$
We note that we do not assume that these quantities are integer numbers. 

Furthermore, for every integer $j \in [j_0, j_1]$
we define $\cP_j$ as the set of primes in the interval $\left[R_j , R_{j+1}\right)$ and then we also 
define the set
$$
\cQ_j = \left\{m \in \left[1, M_j \right]~:~ m\ \text{has no prime factors in}\  \bigcup_{i \le j} \cP_j\right\}.
$$
We note that, by the prime number theorem (with an explicit bound on the error term, we do not however 
need the full power of the current knowledge such as~\cite[Corollary~8.30]{IwKow}), we have the following bound on the cardinality of $\cP_j$, for every 
  $j \in [j_0, j_1]$: 
\begin{equation}
\label{eq: Pj Card}
\# \cP_j \le R_j \(\frac{1}{j} + \frac{1}{\alpha j^2} + O\(\exp\(- \sqrt{\alpha j}\)\)\) \ll    \frac{1}{j} R_j, 
\end{equation}
see~\cite[Equation~(2.8)]{BSZ}, where we have also used that $\alpha j \ge \alpha j_0 \gg 1$. 

As in the proof of~\cite[Theorem~2]{BSZ} we notice that 
the products of the $mr$ with $r\in \cP_j$, $m\in \cQ_j$ for some  $j \in [j_0, j_1]$
are pairwise distinct and obviously belong the interval $[1,N]$, so we conclude 
\begin{equation}
\label{eq: sum PjQj Card}
\sum_{  j_0 \le j \le j_1} \# \cP_j  \# \cQ_j \le N, 
\end{equation}
which is also used in the derivation of~\cite[Equations~(2.20) and~(2.21)]{BSZ}. 

Furthermore, using~\eqref{eq: Pj Card} and recalling choice of the parameters~\eqref{eq: jj}, 
we obtain 
\begin{align*}
\sum_{ j_0 \le j \le j_1} \# \cP_j & \ll 
\sum_{  j_0 \le j \le j_1}\frac{1}{j}  (1+\alpha)^j  \le  (1+\alpha)^{j_1}   \sum_{  j_0 \le j \le j_1}\frac{1}{j} \\
&  \leq   (1+\alpha)^{j_1} \log (j_1/j_0)   \le  \exp\(\alpha j_1\) \log j_0.
\end{align*}
Hence  
\begin{equation}
\label{eq: sum Pj Card}
\sum_{ j_0 \le j \le j_1} \# \cP_j  \ll \exp\(1.5 \alpha^{-1} \(\log (1/\alpha)\)^6\), 
\end{equation}
 provided that $\alpha$ is sufficiently small. 

Now to establish the desired result, we recall that by~\cite[Equation~(2.16)]{BSZ}
\begin{equation}
\label{eq: W}
 \sum_{ n \leq N}\nu(n)F(n) \ll  \sum_{  j_0 \le j \le j_1} W_j + \alpha N, 
\end{equation}
where 
$$
W_j = \sum_{m \in   \cQ_j } \left| \sum_{r \in \cP_j} \nu(r)F(mr)\right|, \qquad  j_0 \le j \le j_1. 
$$

Using the Cauchy--Schwarz inequality, extending the range of summation over $m$ to all positive integers up to  $M_j$, 
 changing the order of summation and recalling that $|\nu(n)|\le 1$, we obtain 
\begin{equation}
\label{eq: W2}
W_j^2 \le  \# \cQ_j   \sum_{r,s \in  \cP_j}  \left| \sum_{m \le M_j} F(mr)\overline{F(ms)}\right|, \qquad  j_0 \le j \le j_1, 
\end{equation}
see~\cite[Equation~(2.17)]{BSZ}. 

The contribution $T_{1,j}$ to the right hand side of~\eqref{eq: W2} from the diagonal terms 
can estimated as in~\cite[Equation~(2.20)]{BSZ} by 
\begin{equation}
\label{eq: T1}
T_{1,j} \ll  M_j \# \cP_j  . 
\end{equation}

To estimate the remaining contribution $T_{2,j}$ from the off-diagonal terms 
we recall our assumption on bilinear sums  with the function $F$. 
More precisely,  splitting the interval of summation into at most $M_j/t$ intervals 
of length $t$ and at most $1$ interval of length $h \le t$, we obtain
$$
\left|\sum_{n \leq M_j} F(nr) \overline{F}(ns)\right| \leq \max\{r,s\}  (M_j/t+1)t \rho .
$$

Hence, using 
(for simplicity) that $r,s \le R_{j+1}\le 2 R_j$ and $M_j R_j \le N$, we derive 
\begin{equation}
\label{eq: T2}
T_{2,j}\le 2 \(\# \cP_j\)^2 R_{j}\(M_j+t\)\rho    \ll   \(\# \cP_j\)^2 \(N +R_j t\)\rho.
\end{equation}

Substituting the bound~\eqref{eq: T1} and~\eqref{eq: T2} in~\eqref{eq: W2}, we obtain 
\begin{align*}
W_j^2 & \ll M_j  \# \cP_j   \# \cQ_j   +  \(\# \cP_j\)^2 \(N/t +R_j\)t\rho  \# \cQ_j\\
&\le  M_j    \# \cP_j  \# \cQ_j    +   N\(\# \cP_j\)^2  \# \cQ_j  \rho +   \(\# \cP_j\)^2   \# \cQ_j R_j  t \rho ,
\end{align*}
which after the substitution in~\eqref{eq: W} implies
\begin{equation}
\label{eq: SSS}
 \sum_{ n \leq N}\nu(n)F(n) \ll    S_1  + S_2   \sqrt{N\rho}   +   S_3  \sqrt{t \rho} + \alpha N, 
\end{equation}
where 
\begin{align*}
S_1 &= \sum_{  j_0 \le j \le j_1} \(M_j  \# \cP_j  \# \cQ_j \)^{1/2} , \\
S_2 & = \sum_{  j_0 \le j \le j_1} \# \cP_j \( \# \cQ_j\)^{1/2} , \\
S_3 & = \sum_{  j_0 \le j \le j_1} \# \cP_j  \(\# \cQ_j   R_j \)^{1/2} .
\end{align*}

To bound the sum  $S_1$, we use the Cauchy--Schwarz inequality and write 
$$
S_1 \ll  \(\sum_{  j_0 \le j \le j_1}     \# \cP_j   \# \cQ_j \)^{1/2}\(\sum_{  j_0 \le j \le j_1} M_j\)^{1/2}.
$$
We estimate the first sum  using~\eqref{eq: sum PjQj Card}, while the second sums 
is easily estimated as 
\begin{align*}
 \sum_{  j_0 \le j \le j_1} M_j &= N   \sum_{  j_0 \le j \le j_1}  (1+\alpha)^{-j-1} \le 
 N   (1+\alpha)^{-j_0-1} \sum_{   j =0}^\infty  (1+\alpha)^{-j}\\
 &= N   \frac{1+\alpha}{\alpha}  (1+\alpha)^{-j_0-1} \ll  N   \frac{1}{\alpha}   \exp(- j_0 \log(1+\alpha) ).
 \end{align*}
Therefore, by the definition of $j_0$ in~\eqref{eq: jj} combined with the inequality  $\log(1+x) \geq x/2$ for $x \in [0,1]$, we obtain 
\begin{equation}
\label{eq: S1}
S_1  \ll    N \exp\(-0.25 \(\log (1/\alpha)\)^3\) .
\end{equation}
Note that~\eqref{eq: S1} is stronger than the bound  recorded in~\cite[Equation~(2.20)]{BSZ}, however this
does not affect the final result as it is dominated by the term $\alpha N$, which is already 
present in~\eqref{eq: SSS}. In particular, we rewrite as 
\begin{equation}
\label{eq: S1 Simple}
S_1  \ll   \alpha  N.
\end{equation}

For the sum $S_2$, writing 
$$
\# \cP_j \( \# \cQ_j\)^{1/2} = \( \# \cP_j \# \cQ_j\)^{1/2}   \( \# \cP_j\)^{1/2},
$$
and applying again the Cauchy--Schwarz inequality,  we obtain
$$
S_2 \le  \(\sum_{  j_0 \le j \le j_1}     \# \cP_j   \# \cQ_j \)^{1/2}\(\sum_{  j_0 \le j \le j_1} \# \cP_j\)^{1/2} .
$$
Now, we see from~\eqref{eq: sum PjQj Card} and~\eqref{eq: sum Pj Card} that 
\begin{equation}
\label{eq: S2}
S_2 \  \ll  N^{1/2} \exp\( \alpha^{-1} \(\log (1/\alpha)\)^6\), 
\end{equation}
We now see that~\eqref{eq: S2}, under the condition~\eqref{eq: Cond1} implies 
\begin{equation}
\label{eq: S2 Simple}
 S_2   \sqrt{N\rho}   \ll   N   \sqrt{\rho}  \exp\( \alpha^{-1} \(\log (1/\alpha)\)^6\)   \ll   \alpha  N .
\end{equation}

Therefore,  it remains to estimate the sum  $S_3$.  We notice that the trivial inequality $\# \cQ_j   R_j \le N$ yields 
$$
S_3  \le N^{1/2}  \sum_{  j_0 \le j \le j_1} \# \cP_j ,
$$ 
which together with~\eqref{eq: sum Pj Card} implies
\begin{equation}
\label{eq: S3}
S_3  \ll  N^{1/2} \exp\(1.5\alpha^{-1} \(\log (1/\alpha)\)^6\). 
\end{equation}
We now see that~\eqref{eq: S3}, under the condition~\eqref{eq: Cond2} implies 
\begin{equation}
\label{eq: S3 Simple}
 S_3    \sqrt{t \rho} \le \sqrt{Nt \rho}  \exp\(1.5\alpha^{-1} \(\log (1/\alpha)\)^6\)   \ll   \alpha  N .
\end{equation}

Substituting the bounds~\eqref{eq: S1  Simple},  \eqref{eq: S2 Simple} and~\eqref{eq: S3  Simple} in~\eqref{eq: SSS}, we derive the desired result. 
\end{proof}

\section{Proof of Theorem~\ref{thm:MobSum}}
We see from Lemma~\ref{lem:SinglSum 2term} that  in   Lemma~\ref{lem:KBSZ} we can take  some $\rho$ with
$$
t^{-1} p^{1/2}  \log p \ll \rho \ll t^{-1} p^{1/2}  \log p.
$$

It is now easy to see that~\eqref{eq: Cond_T1} and~\eqref{eq: Cond_T2}
  ensure the validity of~\eqref{eq: Cond1} and~\eqref{eq: Cond2}, respectively. 
Indeed, using~\eqref{eq: Cond_T1} we see that  for a sufficiently large $p$ we have $\alpha^{-1} \le \log p$
and thus 
$$
 \alpha^{-2} \le (\log p)^2 \mand \(\log (1/\alpha)\)^6\le  ( \log \log p)^6
$$
we derive 
\begin{align*}
 \alpha^{-2} & \exp\( 2 \alpha^{-1} \(\log (1/\alpha)\)^6\)\\ 
 & \qquad \le 
 (\log p)^2 
  \exp\( \frac{2\varepsilon \log p}{3  ( \log \log p)^6}  ( \log \log p)^6\)\\
  & \qquad = p^{2\varepsilon/3}  (\log p)^2  \ll \rho^{-1}.
\end{align*} 
Thus~\eqref{eq: Cond1} holds.

Finally, since $p^{1/2}  \log p \gg t \rho$ then~\eqref{eq: Cond_T2} implies~\eqref{eq: Cond2}
provided that $\alpha$ is small enough and the result follows.

\section{Further perspectives} 

We also note that our results behind the estimates of  Theorem~\ref{thm:MobSum}, in 
particular Lemma~\ref{lem:KBSZ} below
can be applied to estimating exponential sum along sequences with other arithmetic constraints such as 
{\it square-freeness\/} (in which case one can expect stronger results) or {\it smoothness\/}. 

Furthermore, one can also obtain essentially the same results
for analogues of the sums~\eqref{eq:Sum S}  with sequences of the form $r\(\xi_n\)$ with a 
rational function $r(X) \in \Q(X)$. The same approach also works, without 
any charges,  for sums of multiplicative characters with $r\(\xi_n\)$. 

One can also apply our approach to other dynamical systems such as polynomial 
dynamical systems $x \mapsto f(x)$ for a polynomial $f\in \F_p[X]$ of a fixed degree $d\ge 2$  or to monomial  dynamical  systems
$x \mapsto x^e$ for an integer $e\ge 1$  with $\gcd(e,p-1)=1$, which however can be rather large in terms of $p$. 

We also hope that a similar approach may be applied to estimating exponential and character sums along squares, that is, with $\xi_{n^2}$. This kind of questions has also been introduced by  Bourgain~\cite{Bour0,Bour1,Bour1.7}.  Note that for the
dynamical system
 $x \mapsto gx$ over  $\F_p$ such a bound is  given in~\cite{FHS} (which one can 
 probably improve using some arguments from~\cite{OstShp2}).

 \section*{Acknowledgement}
 
The authors are grateful to Winnie Li for her help with explaing  how  Lemma~\ref{lem:Li} follows from a combination of the arguments of~\cite{Li15} and~\cite{Li2}. The authors also would like to thank Stam Nicolis for attracting out attention to connections with theoretical physics. 

 This work started during a very enjoyable visit by the second author to the 
 University of Rouen, whose support and hospitality are gratefully acknowledged. 
 
During the preparation of this work the second author was supported in part by 
the  Australian Research Council  Grants DP170100786 and DP180100201.

\end{document}